\documentclass[a4paper,11pt]{amsart}

\usepackage{cite}
\usepackage[colorlinks=true,linkcolor=black,anchorcolor=black,citecolor=black,filecolor=black,menucolor=black,runcolor=black,urlcolor=black]{hyperref}

\newcommand{\D}{\mathbb{D}}
\newcommand{\C}{\mathbb{C}}
\newcommand{\B}{\mathbb{B}}
\newcommand{\N}{\mathbb{N}}

\newcommand{\Dcal}{\mathcal{D}}
\newcommand{\Ecal}{\mathcal{E}}
\newcommand{\Mcal}{\mathcal{M}}
\newcommand{\Hcal}{\mathcal{H}}
\newcommand{\Fcal}{\mathcal{F}}

\newcommand{\ran}{\textrm{\textup{ran }}}
\newcommand{\Lat}{\textrm{\textup{Lat}}}
\newcommand{\AZ}{\textrm{\textup{AZ}}}

\newcommand{\eps}{\varepsilon}

\newcommand{\comment}[1]{}

\usepackage[english]{babel}

\usepackage{amsthm}
\usepackage{amsmath}
\usepackage{amssymb}

\usepackage{graphicx}

\usepackage[T1]{fontenc}
\usepackage{ae}
\usepackage{enumerate}
\usepackage{stackrel}
\usepackage{tikz-cd}

\theoremstyle{plain}
\newtheorem{thm}{Theorem}
\newtheorem{cor}[thm]{Corollary}
\newtheorem{lem}[thm]{Lemma}
\newtheorem{prop}[thm]{Proposition}

\theoremstyle{definition}

\theoremstyle{remark}
\newtheorem{rem}[thm]{Remark}

\begin{document}


\title{Spectra of Quotient Modules}

\author{Michael Didas}
\address{Schloss Dagstuhl -- Leibniz-Zentrum f\"ur Informatik GmbH, Oktavie-Allee, 66687 Wadern, Germany}
\email{michael.didas@dagstuhl.de}
\author{J\"org Eschmeier}
\address{Fachrichtung Mathematik, Universit\"at des Saarlandes, 66123 Saarbr\"ucken, Germany}
\email{eschmei@math.uni-sb.de}
\author{Michael Hartz}
\address{Fachrichtung Mathematik, Universit\"at des Saarlandes, 66123 Saarbr\"ucken, Germany}
\email{hartz@math.uni-sb.de}
\thanks{MH and MS were partially supported by the Emmy Noether Program of the German Research Foundation (DFG Grant 466012782).}
\author{Marcel Scherer}
\address{Fachrichtung Mathematik, Universit\"at des Saarlandes, 66123 Saarbr\"ucken, Germany}
\email{scherer@math.uni-sb.de}

\date{\today}


\maketitle

\begin{abstract}
  We determine the Taylor spectra of quotient tuples of the $d$-shift on Drury-Arveson spaces
  with finite-dimensional coefficient spaces.
  We show the the Taylor spectrum can be described in terms of the approximate zero
  set of the annihilator ideal, and in terms of the pointwise behavior of the inner multiplier
  associated with the quotient tuple.
\end{abstract}

\section{Introduction and main results}
\label{sec:main-results}

Let $\Dcal$ be a complex Hilbert space and let $M \in \Lat(M_z, H^2_d(\Dcal))$ be a closed invariant subspace for the tuple $M_z=(M_{z_1},\ldots, M_{z_d}) \in L(H^2_d(\Dcal))^d$ consisting of the multiplication operators with the coordinate functions on the $\Dcal$-valued Drury-Arveson space $H^2_d(\Dcal)$ over the Euclidean unit ball $\B_d \subset \C^d$.

Quotient tuples of the form $M_z/M$ on $H^2_d(\mathcal{D}) / M$ appear as model operators for pure commuting
row contractions; see Section \ref{sec:applications} below for more details.
The aim of this note is to provide descriptions of the Taylor spectrum of the quotient tuple $M_z /M$ in the case of finite-dimensional $\Dcal$.
Towards a precise formulation of our main result, we define the approximate zero set of a function $f:\B_d \to \C$ to~be 
\[
  \AZ(f) = \{\lambda \in \overline{\mathbb{B}}_d; \liminf_{z \to \lambda} |f(z)| = 0 \}.
\]
Equivalently, $\lambda \in \AZ(f)$ if and only if there exists a sequence $(z_k)_{k \ge 0}$ in $\mathbb{B}_d$
such that $\lim_{k \to \infty} z_k = \lambda$ and $\lim_{k \to \infty} f(z_k) = 0$.
Moreover, with each $M \in \Lat(M_z, H^2_d(\Dcal))$ we associate a closed ideal
\[
I(M) = \{ f \in \Mcal(H^2_d);\  fH^2_d(\Dcal) \subset M \}
\]
of the multiplier algebra $\Mcal(H^2_d)$. It turns out that the Taylor spectrum in the quotient module $H^2_d(\mathcal{D})/M$ can be expressed in terms of this so-called annihilator ideal $I(M)$:
\begin{thm}\label{thm:main-1}
Let $\mathcal D$ be a finite-dimensional complex Hilbert space and $M \in \Lat(M_z,H^2_d(\Dcal))$. Then, for the Taylor spectrum of the tuple induced by $M_z$ on $H^2_d(\Dcal)/M$, we have 
\[
\sigma(M_z,H^2_d(\Dcal)/M) = \bigcap_{f \in I(M)} \AZ(f).
\]
Moreover, the Taylor spectrum and the right spectrum coincide in this case.
\end{thm}
\noindent This result is motivated by recent work of Clou\^{a}tre and Timko~\cite{CT}, which in particular contains
the equality of the Taylor spectrum and the approximate zero set in the case of one-dimensional $\mathcal{D}$.
The inclusion ``$\subset$'' for finite-dimensional $\mathcal{D}$ can be deduced from the results of Clou\^atre and Timko (see the discussion preceding Proposition~\ref{prop:upper-estimates} below for details). This relies on the corona theorem for $H^2_d$. Related spectral inclusion theorems can also be found in \cite{DE}.

For the reverse inclusion, we establish an alternative description of the spectrum in terms of 
an operator-valued multiplier that generates $M$. To be more precise, if $M \in \Lat(M_z, H^2_d(\Dcal))$, then by the McCullough-Trent version of Beurling's invariant subspace theorem (Theorem 4.1 in \cite{MT}), there exist a Hilbert space $\Ecal$ and a holomorphic multiplier $\theta: \B_d \to L(\Ecal, \Dcal)$ from $H^2_d(\mathcal E)$ to
$H^2_d(\mathcal D)$ such that 
\[
M = \theta H^2_d(\Ecal)
\]
and $\theta$ is inner, which means by definition that the induced multiplication operator $M_\theta: H^2_d(\Ecal) \to H^2_d(\Dcal)$ is a partial isometry. A result of Greene, Richter and Sundberg (Theorem~3.2 in~\cite{GRS}) then guarantees that for almost every $z \in \partial\B_d$ the non-tangential boundary value $\theta(z) \in L(\Ecal, \Dcal)$ exists (in the SOT) and is a partial isometry.

To formulate our result appropriately, we need the following generalized notion of pointwise surjectivity for operator-valued maps:
Given a holomorphic operator-valued map $\theta: \B_d \to L(\Ecal, \Dcal)$, we say that $\theta$ is surjective at $\lambda \in \overline{\B}_d$ if either $\lambda \in \B_d$ and $\theta(\lambda)\Ecal=\Dcal$, or $\lambda \in \partial\B_d$ and there exists an extension of $\theta$ to a holomorphic map
\[
\widehat \theta: U \to L(\Ecal, \Dcal) \text{\quad on some open set \quad} U \supset \B_d \cup \{ \lambda \}
\]
such that $\widehat \theta (\lambda) \Ecal = \Dcal$. Our proof of the reverse inclusion ``$\supset$'' in Theorem~\ref{thm:main-1} relies on the following result of independent interest.
\begin{thm}\label{thm:main-2}
Let $\mathcal D$ be a finite-dimensional Hilbert space, $M \in {\rm Lat}(M_z,H^2_d(\mathcal D))$ and
$\theta:\B_d \rightarrow L(\mathcal E,\mathcal D)$ an inner multiplier from $H^2_d(\mathcal E)$ to
$H^2_d(\mathcal D)$ with $M = \theta H^2_d(\mathcal E)$. Then
\[
\sigma(M_z,H^2_d(\mathcal D)/M) = \{\lambda \in \overline{\mathbb B}_d;\; \theta \mbox{ is not surjective at } \lambda \}.
\]
\end{thm}
\noindent
The inclusion ``$\subset$'' is established in Proposition \ref{prop:upper-estimates}.
The reverse inclusion is finally settled as Corollary~\ref{cor:right-spectrum-as-non-surjectivity-set}.
The main ingredients in the proof are a result of Greene~\cite{G} (to handle the part inside $\B_d$) and structure theory of pure row contractions (in particular their characteristic function~\cite{BES}) applied to $T = P_{M^\bot}M_z|M^\bot$. We will also see that the Taylor spectrum $\sigma(M_z, H^2_d(\mathcal{D})/M)$ agrees
with the right spectrum $\sigma_{r}(M_z, H^2_d( \mathcal{D})/M)$.

Note that the set appearing on the right-hand side in the statement of the preceding theorem extends the classical notion of support of an inner function $\theta: \D \to \C$ on the unit disc: Recall that $\lambda \in \overline{\D}$ belongs to $\text{supp}(\theta)$ if either $\theta(\lambda) = 0$ or $\theta$ does not holomorphically extend across~$\lambda$. This concept has also been one of the starting points for \cite{CT}, but was generalized in another direction there.

In the scalar-valued case $\Dcal = \C$, the set of points $\lambda \in \B_d$ where $\theta(\lambda) : \Ecal \to \C$ is not surjective, is easily seen to coincide with the common zero set $Z(M)$ of all functions in $M$, thus the statement of Theorem~\ref{thm:main-2} then specializes to
\[
\sigma(M_z, H^2_d/M) = Z(M) \cup S(M)
\]
with $S(M)=\{ \lambda \in \partial \B_d;\ \theta \text{ not surjective at }\lambda \}$.
The fact that $\sigma(M_z/M) \cap \B_d = Z(M)$ was first observed by Gleason, Richter and Sundberg in~\cite{GRS}.
It was conjectured in \cite{RS} that $S(M) = \{\lambda \in \partial \B_d;\ \liminf_{z\rightarrow \lambda} \|\theta(z)\| = 0\}$.
This equality would follow if the corona theorem of Costea, Sawyer and Wick \cite{CSW}
held for bounded row multipliers. Since this is not known,
we must leave the question of Gleason, Richter and Sundberg open here.

\section{Calculating the spectrum inside \texorpdfstring{$\B_d$}{B\_d}}
We begin by setting up the necessary notation from multivariable spectral theory. Let $T\in L(X)^d$ be a commuting $d$-tuple of operators on a complex Banach space $X$.
We write $K_\bullet(T,X)$ for the Koszul complex
\[
0 \stackrel{}{\longrightarrow} \Lambda^d(X)
\stackrel{\delta_{d,T}}{\longrightarrow} \Lambda^{d-1}(X)
\stackrel{\delta_{d-1,T}}{\longrightarrow} \ldots
 \stackrel{\delta_{2,T}}{\longrightarrow} \Lambda^1(X)
 \stackrel{\delta_{1,T}}{\longrightarrow} \Lambda^0(X)
 \stackrel{}{\longrightarrow} 0
\]
consisting of the spaces $K_p(T,X) = \Lambda^p(X) = X \otimes \bigwedge^p\C^d \cong X^{\binom{d}{p}}$ and the boundary maps defined by the formula
\[
\delta_{p,T}(x \otimes e_I) = \sum_{\alpha=1}^p (-1)^{\alpha-1} T_{i_\alpha}x \otimes e_{I_\alpha} \quad (x \in X, \ e_I = e_{i_1}\wedge \ldots \wedge e_{i_p}),
\]
where $I=(i_1, \ldots, i_p) \in \N^p$ is a multi-index with $i_1< i_2 < \ldots < i_p$.
Here $\bigwedge^p\C^d$ stands for the $p$-fold exterior product of $\C^d$ with itself, $(e_1, \ldots, e_d)$ is the standard basis of $\C^d$ and the multi-index $I_\alpha\in \N^{p-1}$ arises from $I \in \N^p$ by dropping the $\alpha$-th entry.

The Taylor spectrum  of $T$ (and its various subsets) are explained in terms of the homology groups of $K_\bullet(T,X)$, 
\[
H_p(T,X) = \ker \delta_{p,T} / \ran \delta_{p+1,T} \quad (p=0, \ldots, d).
\]
The Taylor spectrum of $T$ is defined as the set of points in $\C^d$ for which the Koszul complex of $\lambda-T$ is not exact, i.e.,
\[
\sigma(T) = \{\lambda \in \C^d;\ H_p(\lambda-T,X) \neq 0 \text{ for some } p \in \{1, \ldots, d\}\},
\]
where $\lambda - T$ stands for the operator tuple with entries $\lambda_i\cdot 1_X - T_i$ $(1\leq i \leq d)$.
It is well known that $\sigma(T) \subset \C^d$ is compact. 
As usual, we write $\rho(T) = \C^d\setminus \sigma(T)$ for the resolvent set.
A particular role for our calculations is played by the right spectrum
\[
\sigma_r(T) = \{\lambda \in \C^d;\ H_0(\lambda-T,X) \neq 0 \}.
\]
The right essential spectrum $\sigma_{re}(T)$ consists of all $\lambda \in \sigma_r(T)$ where even $\dim H_0(\lambda-T,X) = \infty$.
Note that, modulo the identifications $\Lambda^0(X) \cong X$ and $\Lambda^1(X)\cong X^d$, we have
\[
\delta_{1,\lambda-T} (x_i)_{i=1}^d = \sum_{i=1}^d (\lambda_i - T_i)x_i \qquad ((x_i)_{i=1}^d \in X^d),
\]
and therefore $H_0(\lambda-T,X) \cong X / \sum_{i=1}^d (\lambda_i - T_i)X$. Similarly, up to isomorphy, $\delta_{d,T}$ acts~as
\[
\delta_{d,T}\,x = (T_ix)_{i=1}^d \qquad (x \in X),
\]
and hence $H_d(\lambda-T,X) \cong \bigcap_{i=1}^d \ker(\lambda_i-T_i)$.

We recall a result of Devin Greene \cite{G}  which leads to a description of the points in the Taylor spectrum
of the quotient tuple $M_z/M$ in $\mathbb{B}_d$.
This result relates the homology of the Koszul complex of a multiplication tuple to the
homology of a localized resolution.

Let $\Dcal$ be complex Hilbert space. Given an $M_z$-invariant subspace $M \in \Lat(M_z, H^2_d(\Dcal))$, we apply the McCullough-Trent version of Beurling's invariant subspace theorem (Theorem~4.1 in \cite{MT}) inductively to obtain Hilbert spaces $\Dcal_i$ $(i\geq0)$ starting with $\Dcal_0 = \Dcal$ together with multipliers $\theta_i \in \Mcal(H^2_d(\Dcal_i), H^2_d(\Dcal_{i-1}))$ for $i\geq 1$ such that the induced multiplication operators form an exact sequence
\[
\ldots \longrightarrow H^2_d(\Dcal_2) \stackrel{M_{\theta_2}}{\longrightarrow}
H^2_d(\Dcal_1) \stackrel{M_{\theta_1}}{\longrightarrow}
H^2_d(\Dcal) \stackrel{q}{\longrightarrow} H^2_d(\Dcal)/M \to 0.
\]
Localizing the right-truncated sequence to a point $\lambda \in \B_d$, we obtain a complex
\[
\ldots \longrightarrow \Dcal_2 \stackrel{\theta_2(\lambda)}{\longrightarrow}
\Dcal_1 \stackrel{\theta_1(\lambda)}{\longrightarrow}
\Dcal \longrightarrow 0
\]
denoted by $(\Dcal_\bullet, \theta_\bullet(\lambda))$.
The following result is due to Greene~\cite{G}. For completeness sake,
we indicate a shortened version of the original proof based on standard homological algebra.

\begin{thm}\label{thm:homology}
Given $\lambda \in \B_d$ and $M \in \Lat(M_z, H^2_d(\Dcal))$ for some complex Hilbert space $\Dcal$, there are vector space isomorphisms
\[
H_p(\lambda-M_z, H^2_d(\Dcal)/M) \cong H_p(\Dcal_\bullet, \theta_\bullet(\lambda)) \quad (p\geq 0).
\]
\end{thm}
\begin{proof}
Let $\eps_\lambda : H^2_d(\Dcal) \to \Dcal,\,f \mapsto f(\lambda),$ denote the point evaluation at $\lambda$. It is well known that the augmented Koszul complex
\[
K_\bullet(\lambda-M_z, H^2_d(\Dcal)) \stackrel{\eps_\lambda}{\longrightarrow} \Dcal \longrightarrow 0
\]
is exact in the case $\Dcal = \C$ (see, e.g.,~\cite[Proposition~2.6]{GlRS}). Since tensoring with $1_\Dcal$ preserves exactness, it remains exact in the general case.

We consider the double complex $K =(K_{p,q}, \partial', \partial'')$
with spaces
$K_{p,q} = K_q(\lambda - M_z,H^2_d(\mathcal D_p))$, $p$-th row $(K_{p,\bullet},\partial''_{\bullet})$ 
equal to $(-1)^p$ times the augmented
Koszul complex of the commuting
tuple $\lambda - M_z \in L(H^2_d(\mathcal D_p))^d$ 
and $q$-th column $(K_{\bullet,q},\partial'_{\bullet})$ given by the
$\binom{n}{q}$-fold direct sum of the complex $(H^2_d(\mathcal D_{\bullet}),M_{\theta_{\bullet}})$,
respectively $(\mathcal{D}_{\bullet}(\theta_{\bullet}(\lambda))$ as the last column:
\[
 \begin{tikzcd}
\vdots \arrow[d, ""] & \vdots \arrow[d, ""] & \\
 (-1) \cdot K_\bullet(\lambda-M_z, H^2_d(\Dcal_2)) \arrow[d, "\theta_2"] \arrow[r,"-\eps_\lambda"] & \Dcal_2 \arrow[r,""] \arrow[d,"\theta_2(\lambda)"] & 0 \\
 \phantom{(-1)\cdot} K_\bullet(\lambda-M_z, H^2_d(\Dcal_1)) \arrow[d, "\theta_1"] \arrow[r,"\eps_\lambda"] & \Dcal_1 \arrow[r,""] \arrow[d,"\theta_1(\lambda)"] & 0 \\
 (-1)\cdot K_\bullet(\lambda-M_z, H^2_d(\Dcal)) \arrow[d, "q"] \arrow[r,"-\eps_\lambda"] & \Dcal \arrow[r,""] \arrow[d,"0"] & 0 \\
\phantom{(-1)\cdot} K_\bullet(\lambda-M_z, H^2_d(\Dcal)/M) \arrow[r, ""] \arrow[d, ""]& 0 & \\
0 & & 
\end{tikzcd}
\]
Then $K$ is a double complex with anti-commuting squares and bounded diagonals, and all but the last column and all but the last row are exact. In this setting,
standard double complex arguments (Lemma A2.6 in \cite{EP}) show that there are induced vector space isomorphisms
\[
H_p(\lambda-M_z,H^2_d(\Dcal)/M) \cong H''_p H'_0(K) \cong H'_p H''_0(K) = H_p(\Dcal_\bullet, \theta_\bullet(\lambda)),
\]
as we claimed.
\end{proof}
\noindent As an immediate consequence, we have:
\begin{cor}\label{cor:right-spectrum-inner-points}
Let $\Dcal$ be a complex Hilbert space, $M \in \Lat(H^2_d(\Dcal))$, and $\theta :\B_d \to L(\Ecal, \Dcal)$ be an inner multiplier from $H^2_d(\Ecal)$ to $H^2_d(\Dcal)$ with $M=M_\theta H^2_d$. Then we have 
\[
\sigma_r(M_z, H^2_d(\Dcal)/M) \cap \B_d = \{ \lambda \in \B_d : \theta(\lambda)\Ecal \not= \Dcal \}.
\]
Moreover, if $\mathcal{D}$ is finite-dimensional, then $\sigma_{re}(M_z, H^2_d(\mathcal{D}) / M) \subset \partial \mathbb{B}_d$.
\end{cor}
\begin{proof}
Note that we may choose $\theta_1 = \theta$ and $\Dcal_1 = \Ecal$ in the preceeding theorem to obtain
for $\lambda \in \mathbb{B}_d$ that
\[ H_0(\lambda - M_z, H^2_d(\mathcal{D}) / M) \cong H_0(\mathcal{D}_{\bullet}, \lambda_{\bullet}(\lambda))
  \cong \mathcal{D}/(\theta(\lambda) \mathcal{E}).
\]
This implies the statement for the right-spectrum, as well
as $\sigma_{re}(M_z, H^2_d(\mathcal{D}) / M) \cap \mathbb{B}_d = \emptyset$
if $\mathcal{D}$ is finite-dimensional.
\end{proof}

\section{Upper estimates for the spectrum}
The aim of this section is to provide a proof of both inclusions ``$\subset$'' from the statements of our main Theorems \ref{thm:main-1} and~\ref{thm:main-2}.
As a preparatory result, we state the following observation which can be seen as a partial extension of a result of Sz.-Nagy and Foia\c{s} (Theorem VI.5.2 in \cite{SZF}) to the multivariable case.

\begin{lem}\label{lem:annihilator-determinant}
Let $\Dcal$ be a finite-dimensional Hilbert space with orthonormal basis $(d_1,\ldots,d_N)$ and
$M \in {\rm Lat}(M_z,H^2_d(\mathcal D))$ a closed invariant subspace. Let $\mathcal E$ be a Hilbert space
and $\theta: \B_d \rightarrow L(\mathcal E,\mathcal D)$ a multiplier from $H^2_d(\Ecal)$
into $H^2_d(\Dcal)$ with $\theta H^2_d(\Ecal) \subset M$. Fix vectors $e_1,\ldots,e_N \in \Ecal$
and denote by $\Theta = (\theta_{ij})_{1 \leq i,j \leq N} \in M_N(\Mcal(H^2_d))$ the matrix whose
coefficients are determined by
\[
\theta(z)e_j = \sum\limits^N_{i=1} \theta_{ij}(z)d_i\quad (j=1,\ldots,N, \ z \in \B_d).
\]
Then $\det(\Theta) \in I(M)$. If $\lambda \in \overline{\B}_d$ is a point such that $\theta$ is surjective at $\lambda$,
then there is a multiplier $f \in I(M)$ with $\lim_{\substack{z\to \lambda \\ z\in\B_d}}f(z) = 1$.
\end{lem}

\begin{proof}
Choose a matrix $R = (r_{ij})_{1 \leq i,j \leq N} \in M_N(\Mcal(H^2_d))$ such that
$\Theta R = \det(\Theta)\cdot 1_N$. (It is standard linear algebra that $R$ can be obtained pointwise as the transpose of the so-called cofactor matrix $C$ of $\Theta$, whose components consist -- except for the sign -- of the determinants of all possible $(N-1)\times (N-1)$ submatrices of $\Theta$.)

To prove the first assertion, we may suppose that $\det(\Theta)$ does not vanish identically on $\B_d$.
Then the vectors $e_1,\ldots,e_N$ form a basis of their linear span $\Fcal$. It is elementary to check that the composition of operators
\[
R:\; H^2_d(\mathcal D) \rightarrow H^2_d(\Fcal), \; \sum_{i=1}^N f_i d_i \mapsto \sum_{i=1}^N\Big(\sum_{j=1}^N r_{ij} f_j\Big)e_i
\]
and
\[
\theta: \; H^2_d(\Fcal) \rightarrow H^2_d(\Dcal), \; \sum_{i=1}^N g_i e_i \mapsto
\sum_{i=1}^N\Big(\sum_{j=1}^N \theta_{ij} g_j\Big)d_i = \theta \sum_{j=1}^N g_j e_j.
\]
satisfies $\det(\Theta) f = \theta R f \in M$ for all $f \in H^2_d(\Dcal)$, i.e., $\det(\Theta) \in I(M)$, as desired.

For the remaining part of the assertion, fix $\lambda \in \overline{\B}_d$ and a holomorphic extension $\widehat{\theta}: U \to L(\Ecal,\Dcal)$ of $\theta$ to $U\supset \B_d \cup\{\lambda\}$ such that $\widehat{\theta}(\lambda)\Ecal = \Dcal$.
Then there are vectors $e_1,\ldots,e_N \in \Ecal$ with
\[
\widehat{\theta}(\lambda)e_j = d_j \quad (j=1,\ldots,N).
\]
Let $\Theta=(\theta_{ij})_{1\leq i,j \leq N}$ be the matrix formed as above with respect to the vectors $e_1,\ldots,e_N$ chosen in this way. Then $\Theta$, viewed as a map $\B_d \to M_N(\C)$, continuously extends to $U$ and satisfies
\[
  \lim_{\substack{z \to \lambda \\ z\in \B_d}} \Theta(z) = 1_N.
\]
Hence $f = \det(\Theta)$ defines a function in $I(M)$ as in the statement of the lemma.
\end{proof}

\noindent Now we prove the announced inclusions. The first one can be deduced from a result of Clou\^{a}tre and Timko~\cite[Corollary~3.14]{CT} that depends on the corona theorem for~$H^2_d$ due to Costea, Sawyer and Wick \cite{CSW}.
Alternatively, we can argue directly with the help of the corona theorem.

\begin{prop}\label{prop:upper-estimates}
Let $\Dcal$ be a finite-dimensional Hilbert space and
$M \in {\rm Lat}(M_z,H^2_d(\mathcal D))$ a closed invariant subspace. Let $\mathcal E$ be a Hilbert space
and $\theta: \B_d \rightarrow L(\mathcal E,\mathcal D)$ an inner multiplier from $H^2_d(\Ecal)$
into $H^2_d(\Dcal)$ with $\theta H^2_d(\Ecal) \subset M$. Then we have the inclusions
\[
\sigma(M_z,H^2_d(\Dcal)/M) \subset \bigcap_{f \in I(M)} \AZ(f) \subset \{\lambda \in \overline{\mathbb B}_d;\; \theta \mbox{ is not surjective at } \lambda \}
\]
\end{prop}
\begin{proof}
Note that the second inclusion readily follows from the preceding lemma which says that, if $\lambda$ does not belong to the set on the right-hand side, then there is a function $f \in I(M)$ with $\lambda \not \in AZ(f)$.


Towards a proof of the first inclusion, let $\lambda \notin \bigcap_{f \in I(M)} \AZ(f)$.
Then there exists $h \in I(M)$ with $\lambda \notin \AZ(h)$, hence
\begin{equation*}
  \inf_{z \in \mathbb{B}_d} \sum_{i=1}^d |\lambda_i - z_i| + |h(z)| > 0.
\end{equation*}
By the corona theorem for $H^2_d$ \cite{CSW}, there exist $f_1,\ldots,f_d,f \in \mathcal{M}(H^2_d)$ such that $\sum_{i=1}^d (\lambda_i - z_i) f_i + f h = 1$. Let $g = f h \in I(M)$.
From the very definition of $I(M)$ it follows that $M_g/M = 0$ and hence
\[
\sum_{i=1}^d (\lambda_i - M_{z_i}/M) (M_{f_i}/M) = 1_{H^2_d(\mathcal D)/M}.
\]
Lemma 2.2.4 in \cite{EP} shows that $\lambda \notin \sigma(M_z,H^2_d(\mathcal D)/M)$ as desired.
\end{proof}

\section{Lower estimate for the spectrum and proof of main results}
In view of Proposition~\ref{prop:upper-estimates}, both Theorems 1 and 2 will follow as soon as we can show the missing inclusion
\[
\{\lambda \in \overline{\mathbb B}_d;\; \theta \mbox{ is not surjective at } \lambda \}
\subset \sigma(M_z/M).
\]
To achieve this, we make use the characteristic function $\theta_T$ of the pure row contraction $T = P_H M_z|H \in L(H)^d$ where $H = M^\perp$.
Let us first establish the necessary notations and recall some basic facts.

Let $\Hcal$ be a complex Hilbert space, and let $T\in L(\Hcal)^d$ be a commuting row contraction, which means by definition that $T_1T_1^* + \ldots + T_dT_d^* \leq 1_\Hcal$ or, equivalently, that the row operator
\[
T = [ T_1, \ldots, T_d] : \Hcal^d \to \Hcal, \quad (x_i)_{i=1}^d \mapsto \sum_{i=1}^d T_ix_i,
\]
is a contraction. Note that (modulo the canonical identifications) the row operator $T:\Hcal^d\to \Hcal$ is nothing else than the boundary map $\delta_{1,T}$ in the Koszul complex of $T$. Similarly, the adjoint $T^*:\Hcal\to \Hcal^d$ acts as $\delta_{d,T^*}$.

Following \cite{BES} we define the defect operators 
\[
D_T = (1_{\Hcal^d} - T^*T)^{1/2} \in L(\Hcal^d) \qquad \text{and} \qquad 
D_{T^*} = (1_\Hcal - TT^*)^{1/2} \in L(\Hcal),
\]
and the respective defect spaces as 
\[
\Dcal_T = \overline{D_T\Hcal^d} \subset \Hcal^d \qquad \text{and} \qquad \Dcal_{T^*} = \overline{D_{T^*}\Hcal} \subset \Hcal.
\]
The intertwining relations (Lemma 2.1 in \cite{BES})
\[
T D_T = D_{T^*} T \qquad \text{and} \qquad T^* D_{T^*} = D_T T^*
\]
yield the inclusions $T \Dcal_T \subset \Dcal_{T^*}$ and $T^* \Dcal_{T^*} \subset \Dcal_T$. It is well known (Lemma~2.2 in~\cite{BES}) that the so-called characteristic function of $T$ defined by
\[
\theta_T : \B_d \to L(\Dcal_T, \Dcal_{T^*}), \quad \theta_T(z) = -T +D_{T^*} (1_\Hcal - ZT^*)^{-1} ZD_T
\]
is an analytic function that induces a well-defined contractive multiplier
\[
M_{\theta_T}: H^2_d(\Dcal_T) \to H^2_d(\Dcal_{T^*}), \quad f \mapsto \theta_T f.
\]
Here, the symbol $Z$ stands for row operator $Z = [ z_1 1_\Hcal, \ldots, z_d 1_\Hcal ] : \Hcal^d \to \Hcal$ associated
with $z=(z_1, \ldots, z_d) \in \C^d$.
Details on characteristic functions of commuting row contractions and their properties can be found in \cite{BES}.

Fix $z\in \rho(T) \cap \partial \B_d$. By the poynomial spectral mapping theorem for $T^*$ applied to $p(w) = 1- \sum_{i=1}^d z_iw_i \in \C[w]$, we obtain
\[
0 = 1-|z|^2 \not\in \{ 1 - \langle z, w\rangle; \, w \in \sigma(T) \} = \sigma(1_\Hcal - ZT^*).
\]
Hence the characteristic function $\theta_T$ of $T$ 
extends to a holomorphic map $\widehat{\theta}_T : U \to L(\Dcal_T, \Dcal_{T^*})$ given by the same formula that defines $\theta_T$ on the open set
\[
U = \{ z \in \C^d : 1_\Hcal -ZT^* \text{ invertible} \} \supset \B_d \cup (\rho(T) \cap \partial \B_d) \supset \overline{\B}_d \cap \rho(T).
\]
If $T \in L(\Hcal)$ is a single contraction
with $\rho(T) \cap \partial \D \neq \emptyset$, then the values of the extended characteristic
function  are unitary operators $\widehat{\theta}_T: \; \mathcal D_T \rightarrow \mathcal D_{T^*}$ for each point
$\lambda \in \rho(T) \cap \partial \D$. In particular, $\dim \Dcal_T = \dim \Dcal_{T^*}$ (see
Chapter VI.1 in \cite{SZF}). In the multivariable case the situation is quite different. Nevertheless, we
obtain at least a partial result of the same type.

\begin{thm}\label{thm:charcteristic-function-local-surjectivity}
Let $T \in L(\Hcal)^d$ be a commuting row contraction such that $\dim \Dcal_{T^*} < \infty$. Then the characteristic
function
\[
\theta_T: \; \B_d \rightarrow L(\Dcal_T,\Dcal_{T^*})
\]
of $T$ is surjective at every point $\lambda \in \overline{\B}_d \cap \rho(T)$.
\end{thm}

\begin{proof}
Let $U$ and $\widehat{\theta}_T: \; U \rightarrow L(\mathcal D_T,\mathcal D_{T^*})$ be defined as above and let
$\lambda \in U \cap \rho(T)$ be given. We show that $\widehat{\theta}_T(\lambda)$ is surjective.
Towards this, we first observe that
\[
\widehat{\theta}_T(z)^* = -T^* + D_TZ^*(1_{\Hcal} - TZ^*)^{-1}D_{T^*} \in L(\Dcal_{T^*},\Dcal_T)
\]
for all $z \in U$ and hence
\begin{align*}
D_T \widehat{\theta}_T(z)^* 
& = (-T^* + (1_{\Hcal}-T^*T)Z^*(1_{\Hcal}-TZ^*)^{-1})D_{T^*} \\
& = (-T^* + Z^*(1_{\Hcal}-TZ^*)^{-1} - T^*(TZ^*)(1_{\Hcal}-TZ^*)^{-1})D_{T^*} \\
& = (-T^* + Z^*(1_{\Hcal}-TZ^*)^{-1} + T^* - T^*(1_{\Hcal}-TZ^*)^{-1})D_{T^*} \\
& = (Z^* - T^*)(1_{\Hcal}-TZ^*)^{-1}D_{T^*}
\end{align*}
for $z \in U$. If $y = D_{T^*}x$ $(x \in \Hcal)$, then
\[
D_T \widehat{\theta}_T(z)^*y = (Z^* - T^*)(1_{\Hcal}-TZ^*)^{-1}(1_{\Hcal}-T^*T)x
\]
for $z \in U$. In particular, for $z \in U \cap \rho(T)$, we have that $Z^* - T^* \cong \delta_{d,\overline{z}-T^*}$ is injective, so for $x\in \Hcal$, $y = D_{T^*}x$ with
$\widehat{\theta}_T(z)^*y = 0$, we obtain that
\[
\| y \|^2 = \langle D_{T^*}^2x,x\rangle = \langle(1_{\Hcal}-T^*T)x,x \rangle = 0.
\]
Hence the condition that $\dim \Dcal_{T^*} < \infty$ implies that $\widehat{\theta}_T(z)^* \in L(\Dcal_{T^*},\Dcal_T)$
is injective for $z \in  U \cap \rho(T)$. But then $\widehat{\theta}_T(z) \in L(\Dcal_T,\Dcal_{T^*})$ is
surjective for $z \in  U \cap \rho(T)$.
\end{proof}
\noindent Let $T$ be a row contraction. We say that $T$ is pure if the completely positive map
$P_T: B(\Hcal) \to B(\Hcal), \, X \mapsto \sum_{i=1}^d T_iXT_i^*$ associated with $T$ satisfies
\[
\text{SOT}-\lim_{m\to \infty}P_T^m(1_H) = 0.
\]
For a pure row contraction $T$, the map
\[
j: \Hcal \rightarrow H^2_d(\mathcal D_{T^*}),\quad j(x) = \sum\limits_{\alpha \in \mathbb N^d}
\frac{| \alpha|!}{\alpha!} (D_{T^*}T^{*\alpha}x) z^\alpha
\]
yields an isometry intertwining $T^* \in L(\Hcal)^d$ and $M^*_z \in L(H^2_d(\mathcal D_{T^*}))^d$ componentwise such that
\[
M_{\theta_T}M^*_{\theta_T}+jj^*=1_{H^2_d(\mathcal D_{T^*})},
\]
see \cite[Lemma 3.6]{BES}.
Since $M_{\theta_T}$ is a partial isometry, this leads to the orthogonal direct sum decomposition
\[
H^2_d(\mathcal D_{T^*}) = \theta_TH^2_d(\mathcal D_T) \oplus j\mathcal H.
\]
We will subsequently refer to the map $j$ from above as the canonical dilation of~$T$.

Let us return to our default setting now:
Define $H = H^2_d(\mathcal D) \ominus M$ and $T = P_H M_z|H \in L(H)^d$, which is known to be a pure row contraction. 
In view of Proposition~\ref{prop:upper-estimates},
the following missing inclusion settles the proof of our main results from Section~\ref{sec:main-results},
with the exception of the statement about the right spectrum.

\begin{cor}\label{cor:right-spectrum-as-non-surjectivity-set}
Let $\mathcal D$ be a finite-dimensional Hilbert space, $M \in {\rm Lat}(M_z,H^2_d(\mathcal D))$ and
$\theta:\mathbb B_d\rightarrow L(\mathcal E,\mathcal D)$ an inner multiplier from $H^2_d(\mathcal E)$ to
$H^2_d(\mathcal D)$ with $M = \theta H^2_d(\mathcal E)$. Then
\[
\sigma(M_z,H^2_d(\mathcal D)/M) \supset \{\lambda \in \overline{\mathbb B}_d;\; \theta \mbox{ is not surjective at } \lambda \}.
\]
\end{cor}

\begin{proof}
By Corollary~\ref{cor:right-spectrum-inner-points},
it suffices to show that
\[
\sigma(T)\cap \partial \B_d \supset \{\lambda \in \partial \B_d;\; \theta \mbox{ is not surjective at } \lambda \},
\]
or equivalently, that $\theta$ is surjective at $\lambda$ for all $\lambda \in \rho(T)\cap\partial\B_d$.

Towards this, fix such a point $\lambda$. Then, Theorem~\ref{thm:charcteristic-function-local-surjectivity} guarantees that the characteristic function $\theta_T$ is surjective at $\lambda$. The rest of the proof is about establishing a connection between $\theta_T$ and $\theta$.

Let $\mathcal R \subset H^2_d(\mathcal D)$
be the smallest reducing subspace for $M_z$ with $\mathcal R \supset H$. Then (see \cite[Section 2]{BEK+17})
\[
\mathcal R =\bigvee_{\alpha \in \mathbb N^d} z^{\alpha}(\mathcal R \cap \mathcal D) = H^2_d(\mathcal R \cap \mathcal D).
\]
Since the inclusion map $i: \; H \rightarrow  H^2_d(\mathcal R \cap \mathcal D)$ and the canonical dilation
$j: \; H \rightarrow H^2_d(\Dcal_{T^*})$ are both minimal dilations for $T$, there is a unitary operator
$U: \; \mathcal D_{T^*} \rightarrow \mathcal R \cap \mathcal D$ such that $1 \otimes U \circ j = i$; see \cite[Theorem 3.1]{BEK+17}.
Define $\widehat{\mathcal D} = \mathcal D \ominus (\mathcal R \cap \mathcal D)$. Then
\[
H^2_d(\widehat{\mathcal D}) = H^2_d(\mathcal D) \ominus H^2_d(\mathcal R \cap \mathcal D) =
H^2_d(\mathcal D) \ominus \mathcal R \subset M
\]
is the largest reducing subspace for $M_z$ contained in $M$. Note that
\[
  (1 \otimes U)\theta_T H^2_d(\mathcal D_T) = (1 \otimes U) (H^2_d(\mathcal D_{T^*})\ominus {\rm Im}j)
= H^2_d(\mathcal R \cap \mathcal D)\ominus H = M \cap H^2_d(\widehat{\mathcal D})^{\perp}.
\]
Hence we obtain the orthogonal decomposition
\[
M = H^2_d(\widehat{\mathcal D}) \oplus (M \cap H^2_d(\widehat{\mathcal D})^{\perp}) =
H^2_d(\widehat{\mathcal D}) \oplus (1 \otimes U)(\theta_T H^2_d(\mathcal D_T)).
\]
The operator-valued map $\widetilde{\theta}: \; \mathbb B_d \rightarrow L(\widehat{\mathcal D} \oplus \mathcal D_T,\mathcal D)$,
\[
\widetilde{\theta}(z) = 1_{\widehat{\mathcal D}} \oplus (U \theta_T(z))
\]
defines an inner multiplier from $H^2_d(\widehat{\mathcal D} \oplus \mathcal D_T)$ into $H^2_d(\mathcal D)$ with
$\widetilde{\theta} H^2_d(\widehat{\mathcal D} \oplus \mathcal D_T) = M$.

Since  $\theta_T$ is surjective at $\lambda$, so is $\widetilde{\theta}$.
Known uniqueness results about inner multipliers show that there exists a partial isometry $V: \widehat{\mathcal{D}} \oplus \mathcal{D}_T \to \mathcal{E}$ such that $\widetilde{\theta}(z) = \theta(z) V$
and $\theta(z) = \widetilde{\theta}(z) V^*$ for all $z \in \mathbb{B}_d$;
see \cite[Theorem~4.2]{MT} or \cite[Proposition~2.3]{CHS20}.
The second equality shows that $\theta$ extends to a holomorphic function in a neighborhood of $\lambda$,
and the first equality then shows that $\theta$ is surjective at $\lambda$.
\end{proof}

\begin{rem}
  \label{rem:right_spectrum}
  A general result from multivariable spectral theory (Corollary 3.5 in~\cite{WE}) says that, for a commuting tuple $T\in L(H)^d$ with $\sigma(T) \subset \overline{\B}_d$, we have $\sigma_r(T)\cap \partial\B_d = \sigma(T)\cap \partial\B_d$. Moreover, by Theorem \ref{thm:main-2} and Corollary \ref{cor:right-spectrum-inner-points}, we have
  \begin{equation*}
    \sigma(M_z, H^2_d(\mathcal{D})/M) \cap \mathbb{B}_d = \{\lambda \in \mathbb{B}_d; \theta \text{ is not surjective at } \lambda \} = \sigma_r(M_z, H^2_d(\mathcal{D})/M) \cap \mathbb{B}_d.
  \end{equation*}
  Therefore, in the setting of Theorem \ref{thm:main-2}, we have
  \begin{equation*}
    \sigma(M_z, H^2_d(\mathcal{D})/M)  = \sigma_r(M_z, H^2_d(\mathcal{D})/M).
  \end{equation*}
\end{rem}

\section{Applications to row contractions}
\label{sec:applications}
Since every pure commuting row contraction $T \in L(\mathcal H)^d$ is unitarily equivalent to a quotient tuple of the
form $M_z/M \in L(H^2(\mathcal D_{T^*})/M)^d$, Theorem \ref{thm:main-2} yields a description of the Taylor spectrum of
$T$ in terms of its characteristic function.

\begin{cor}\label{cor:application-1}
Let $T \in L(\mathcal H)^d$ be a pure commuting row contraction such that $\dim \mathcal D_{T^*} < \infty$. Then
\[
\sigma(T)=\sigma_r(T)=\lbrace \lambda\in \overline{\mathbb B}_d; \; \theta_T \mbox{ is not surjective at }\lambda\rbrace.
\]
\end{cor}

\begin{proof}
  Recall from the discussion following the proof of Theorem \ref{thm:charcteristic-function-local-surjectivity} that since $T$ is a pure row contraction, the map
\[
j: \mathcal H \rightarrow H^2_d(\mathcal D_{T^*}), j(x) = \sum\limits_{\alpha n\in \mathbb N^d}
\frac{| \alpha|!}{\alpha!} (D_{T^*}T^{*\alpha}x) z^\alpha
\]
is an isometry intertwining $T^* \in L(\mathcal H)^d$ and $M^*_z \in L(H^2_d(\mathcal D_{T^*}))^d$ componentwise such that
\[
M_{\theta_T}M^*_{\theta_T}+jj^*=1_{H^2_d(\mathcal D_{T^*})}.
\]
Since $M_{\theta_T}$ is a partial isometry, this leads to the orthogonal direct sum decomposition
\[
H^2_d(\mathcal D_{T^*}) = \theta_TH^2_d(\mathcal D_T) \oplus j\mathcal H.
\]
Define $M = \theta_T H^2_d(\mathcal D_T) \in {\rm Lat}(M_z,H^2_d(\mathcal D_{T^*}))$ and $H=H^2_d(\mathcal D_{T*})\ominus M$. Then
via the unitary operator $j: \mathcal H \rightarrow H$ the given tuple $T \in L(\mathcal H)^d$ and the compression
$P_H M_{z|H}\cong M_z/M \in L(H^2(\mathcal D_{T^*})/M)^d$ are unitarily equivalent. Thus the assertion follows from
Theorem \ref{thm:main-2} and Remark \ref{rem:right_spectrum}.
\end{proof}

In the single variable case $d = 1$ there is a natural extension of the result stated in Corollary 8 to the case of
completely non-unitary contractions $T \in L(\mathcal H)$ with no restriction on the defect space $\mathcal D_{T^*}$
(Theorem VI.4.1 in \cite{SZF}). At this moment it remains open whether Corollary~\ref{cor:application-1} remains true without the condition that
the defect space $\mathcal D_{T^*}$ is finite dimensional.

As a consequence of Corollary~\ref{cor:application-1} we obtain a dichotomy for pure commuting row contractions
$T \in L(\mathcal H)^d$ with $\dim \mathcal D_{T^*} < \infty$ whose characteristic function extends to an open
neighbourhood of the closed ball $\overline{\B}_d$.

\begin{cor}
Let $T\in L(\mathcal H)^d$ be a pure commuting row contraction such that $\dim \mathcal D_{T^*}< \infty$. Suppose that its characteristic
function extends to a holomorphic map $\widehat{\theta}_T: U \rightarrow L(\mathcal D_T,\mathcal D_{T^*})$ on an open set
$U \supset \overline{\B}_d$. Then either $\sigma(T) = \overline{\B}_d$ or $\dim \mathcal H < \infty$ and
$\sigma(T) \subset \B_d$ is finite.
\end{cor}

\begin{proof}
As seen in the proof of Corollary \ref{cor:application-1} there is a closed invariant subspace
$M = \theta_TH^2_d(\mathcal D_T) \in {\rm Lat}(M_z,H^2_d(\mathcal D_{T^*}))$ such that $T$ is unitarily equivalent
to the quotient tuple $M_z/M\in L(H^2(\mathcal D_{T^*})/M)^d$. Suppose that $\sigma(T) \neq \overline{\mathbb B}_d$.
Since $\sigma(T)$ is closed, Theorem \ref{thm:charcteristic-function-local-surjectivity} shows there is
a point $\lambda \in \mathbb B_d$ with $\theta_T(\lambda) \mathcal D_T = \mathcal D_{T^*}$. Then Theorem 1.4 in
\cite{GRS} implies that $\widehat{\theta}_T(\lambda) \mathcal D_T = \mathcal D_{T^*}$ for all $\lambda \in \partial \mathbb B_d$.
Corollary~\ref{cor:application-1}, now shows that
$\sigma(T) \subset \mathbb{B}_d$. On the other hand, 
since $\dim \mathcal{D}_{T^*} < \infty$, Corollary \ref{cor:right-spectrum-inner-points}
implies that $\sigma_{re}(T) \subset \partial \mathbb{B}_d$,
hence $\sigma_{re}(T) = \emptyset$.
Therefore, $\dim \mathcal{H} < \infty$ (see e.g.\ Theorems 9 and 17 in \cite[Section 19]{VM}), and hence $\sigma(T) \subset \mathbb{B}_d$ is a finite set.
\end{proof}

\bibliographystyle{plainurl}
\bibliography{bibliography}

\end{document}